\documentclass[12pt]{amsart}
\usepackage{amssymb}
\usepackage[shortlabels]{enumitem}
\usepackage[margin = 1in]{geometry}
\newcommand{\norm}[1]{\lVert#1\rVert}
\newtheorem{theorem}{Theorem}
\newtheorem{lemma}{Lemma}
\theoremstyle{definition}
\newtheorem{definition}{Definition}
\newtheorem{remark}{Remark}
\newtheorem{example}{Example}
\begin{document}
\title[Sub-exponential Rate of Convergence]{Sub-exponential Rate of Convergence to Equilibrium for Processes on the Half-Line}
\author{Andrey Sarantsev}
\address{Department of Mathematics and Statistics, University of Nevada, Reno}
\email{asarantsev@unr.edu}
\dedicatory{This article is dedicated to Professor Anna Panorska\\ on the occasion of her 20 years at the University of Nevada, Reno.}

\begin{abstract}
A powerful tool for studying long-term convergence of a Markov process to its stationary distribution is a Lyapunov function. In some sense, this is a substitute for eigenfunctions. For a stochastically ordered Markov process on the half-line, Lyapunov functions can be used to easily find explicit rates of convergence. Our earlier research focused on exponential rate of convergence. This note extends these results to slower rates, including power rates, thus improving results of (Douc, Fort, Guillin, 2009). 
\end{abstract}

\maketitle

\thispagestyle{empty}

\section{Introduction}

\subsection{Lyapunov functions and long-term convergence} Take a continuous-time Markov process $X = (X(t),\, t \geqslant 0)$ with transition function $P^t(x, A) := \mathbb P(X(t) \in A\mid X(0) = x)$, and  semigroup $P^tf(x) = \mathbb E\left[f(X(t))\mid X(0) = x\right]$. Suppose $X$ has a unique {\it stationary distribution} $\pi$, a probability measure on its state space $S$ such that if the process starts from $X(0) \sim \pi$, then for every $t \geqslant 0$ we have: $X(t) \sim \pi$. Does the measure $P^t(x, \cdot)$ converge to $\pi$ for all $x$ as $t \to \infty$, in which distance, and how fast? 

For continuous-time Markov chains on a finite state space, the answer depends on the generating matrix $\mathcal L$ of transition intensities: Loosely speaking, the eigenvalue closest to $0$ gives us the exponential convergence rate. For general state spaces $S$, for example in $\mathbb R^d$, this description of eigenvalues is difficult or impossible. This is because $\mathcal L$ is no longer a matrix: It is an operator defined by 
$$
\mathcal Lf(x) = \lim\limits_{t \downarrow 0}\frac1t\left[P^tf(x) - f(x)\right]
$$
on a certain space of functions $f : S \to \mathbb R$. It is often hard to find eigenvalues and eigenfunctions of $\mathcal L$, but there is a useful substitute: {\it Lyapunov functions}. There are various versions of this concept, but the basic property is (loosely speaking) that this is a function $V : S \to [1, \infty)$ such that for some compact subset $K \subseteq S$, 
\begin{equation}
\label{eq:Lyapunov}
\mathcal LV(x) \leqslant -c < 0,\, x \in S\setminus K.
\end{equation}
Then under some additional technical conditions the process $X$ is {\it ergodic:} That is, $P^t(x, \cdot) \to \pi$ as $t \to \infty$ for every $x \in S$ in the {\it total variation norm}:
\begin{equation}
\label{eq:TV-def}
\sup\limits_{A \subseteq S}|P^t(x, A) - \pi(A)| \to 0,\, t \to \infty.
\end{equation}
 Let us explain the intuition. The function $V$ measures the ``altitude" reaches by the process $X$. When the process is outside of the ``valley'' which is the compact set $K$, it is compelled by the condition~\eqref{eq:Lyapunov} to ``decrease altitude''. This random process has negative ``drift'' until it gets back to the ``valley''. This ``gravity force'' implies (after additional technical work) its ergodicity.  See articles \cite{MT93a, MT93b} for rigorous exposition, and the classic book \cite{MTBook} for similar concepts on discrete-time Markov chains. Under a stronger than~\eqref{eq:Lyapunov} condition: For a constant $k > 0$ and a compact $K \subseteq S$, 
\begin{equation}
\label{eq:Lyapunov-exp}
\mathcal LV(x) \leqslant -kV(x),\, x \in S\setminus K,
\end{equation}
we get {\it exponential ergodicity:} $P^t(x, \cdot) \to \pi$ as $t \to \infty$ for all $x \in S$ as fast as $Ce^{-\varkappa t}$ for some $C, \varkappa > 0$. See rigorous statements in \cite{DMT95}. However, finding or estimating $\varkappa$ proved to be very hard (compared to finding or estimating $k$ from~\eqref{eq:Lyapunov-exp}, which is often easy in practice); see for example \cite{MT94, RT00}. Among other articles on this topic, let us mention \cite{BCG08, Davies}; and applications to Markov Chain Monte Carlo techniques, \cite{MCMC}.

There is one special case when we simply conclude that $\varkappa = k$ from~\eqref{eq:Lyapunov-exp}: The state space $S = \mathbb R_+ := [0, \infty)$, the exceptional compact set is $K = \{0\}$, and the process $X$ is {\it stochastically ordered}. The latter means that for any $x_2 \geqslant x_1 \geqslant 0$, we can {\it couple} two copies $X_1$ and $X_2$ of this process starting from $X_i(0) = x_i$, on the same probability space, so that $X_1(t) \leqslant X_2(t)$ for all $t \geqslant 0$. This statement was shown in \cite{LMT96} and improved in \cite{Sar16} using the coupling method. We used this method in risk theory \cite{Risk} and for more general processes, so-called Walsh diffusions, which include as a particular case diffusions on the real line, \cite{Walsh}. 

\subsection{Our results} One can also have a condition stronger than~\eqref{eq:Lyapunov}, but  weaker than~\eqref{eq:Lyapunov-exp}: For some increasing concave function $\varphi : [1, \infty) \to [0, \infty)$ and some compact set $K \subseteq S$, 
\begin{equation}
\label{eq:Lyapunov-phi}
\mathcal LV(x) \leqslant -\varphi(V(x)),\, x \in S\setminus K.
\end{equation}
It is shown in \cite{DFG09} that this implies (under some additional technical conditions) a {\it subexponential}, or {\it subgeometric}, convergence rate $P^t(x, \cdot) \to \pi$ to $t \to \infty$. That is, the total variation distance converges to zero, but at a rate slower than exponential, given by $Ce^{-\varkappa t}$. In that article and earlier ones \cite{FR05, TT94}, they were able to find explicit rates, similar to $ct^{-\alpha}$ for some $C, \alpha > 0$, or $C\exp\left[-\lambda t^{\beta}\right]$ for $C, \lambda > 0$ and $\beta \in (0, 1)$.

In this note, as opposed to \cite{DFG09, FR05}, we assume that the compact set $K = \{0\}$ on the state space $S = [0, \infty)$, and that  the process $X$ is stochastically ordered. We improve upon the results of these articles cited above. In other words, we combine ideas from \cite{DFG09, LMT96, Sar16} in the context of subgeometric ergodicity. This leads to simple and elegant proofs. We explore how our new results fit into exponential ergoicity framework from \cite{Sar16}. 

We give examples of reflected diffusions, jump-diffusions, and L\'evy processes on $\mathbb R_+$. These processes behave as, correspondingly, diffusions, jump-diffusions, and L\'evy proceses, as long as they do not hit zero. When this happens, they are reflected back inside the positive half-line. In our examples, they have power rates of convergence: $h(t) := t^{-\beta}$ for $\beta > 0$. These processes  have been extensively studied, from the original article \cite{Skorohod} on stochastic differential equations with reflection on the half-line, to recent articles \cite{AB} on multidimensional reflected jump-diffusions and \cite{M1, M2} on reflected L\'evy processes. These processes have applications in queueing theory (heavy traffic approximation), see \cite{Qnew, Q} and references therein; and, more recently, in financial mathematics, see \cite{Survey, SPT} and references therein.

We prove convergence not only in total variation norm, but in stronger norms. Such convergence implies the convergenc of moments up to a certain order. As in \cite{DFG09}, there is a tradeoff between the norm and the rate of convergence. 

\subsection{Overview of related research} Similar questions were studied with relation to functional inequalities: log-Sobolev, Poincare, Talagrand. The latter inequalities are related to probabilities of large deviations from the median. We cite the books \cite[Chapter 1]{WangBook} and \cite{LedouxBook}. However, they deal with exponential rate of convergence. This is the case of {\it spectral gap:} when the generator has eigenvalues with negative real parts, separated from zero. Here, we deal with subexponential rate of convergence, when there is no spectral gap. There is some literature on this case: see an old article \cite{Rockner} and very recent articles \cite{Ben, Degenerate} and citations therein. But they deal with convergence in $L^2$-norm: If $\pi$ is the stationary distribution, then for $u : S \to \mathbb R$ with $\int_Su^2(x)pi(\mathrm{d}x) < \infty$, 
$$
\int_S\left[(P^tu)(x) - \int_S(P^tu)(y)\,\pi(\mathrm{d}y)\right]^2\,\pi(\mathrm{d}x) \leqslant C(u)f(t),\, f(t) \to 0,\, t \to \infty,
$$
Our estimates are in the total variation norm as in~\eqref{eq:TV-def} or stronger related norms, which take into account the starting point $x$, instead of integrating over $x$ with respect to $\pi$. 

One can also study convergence in Wasserstein distance. This is another family of distance metrics which is different than total variation and other distance measures from this article. Wasserstein distance takes into account the metric structure of the state space, which is related to weak convergence. In fact, convergence of measures in Wasserstein distance is equivalent to weak convergence plus convergence of moments of the given order. The author studied this question for exponential convergence rate for jump-diffusion processes in \cite{Sar}. For subgeometric convergence, this was studied in \cite{Butkovsky}. 

\subsection{Organization of the article} In Section 2, we state all notation and definitions, present main results, and discuss their relationship with existing research. In Section 3, we present examples, and Section 4 is devoted to proofs of main results. 

\subsection{Acknowledgements} We thank Professor Mark M. Meerschaert for invitation to give a talk in November 2016 at the Colloquium at the Department of Statistics \& Probability, Michigan State University in East Lansing, and for useful discussion there. We thank the Department of Mathematics \& Statistics, University of Nevada in Reno, for a supportive atmosphere for research and professional development.

\section{Notation, Definitions, and Main Results}

\subsection{Notation and definitions} Let $\mathbb R_+ := [0, \infty)$. For two probability measures $\rho_1$ and $\rho_2$ on $\mathbb R_+$, we can define their {\it stochastic maximum} $\rho_0 := \rho_1\vee\rho_2$ is defined by $\rho_0((x, \infty)) := \rho_1((x, \infty))\vee\rho_2((x, \infty))$, $x \geqslant 0$. For a signed measure $\nu$ on $\mathbb R_+$ and a function $f : \mathbb R_+ \to \mathbb R$, we define $(\nu, f) := \int_0^{\infty}f(x)\nu(\mathrm{d}x)$. For a function $V : \mathbb R_+ \to [1, \infty)$, the {\it $V$-norm} of a signed measure $\nu$ on $\mathbb R_+$ is  
$$
\norm{\nu}_V := \sup\limits_{|f| \leqslant V}|(\nu, f)|,
$$
where the sup is taken over all functions $f : \mathbb R_+ \to \mathbb R$ such that $|f(x)| \leqslant V(x)$. If $V \equiv 1$, then this norm is denoted by $\norm{\cdot}_{\mathrm{TV}}$ and is called the {\it total variation norm}. A family of Borel measures $(Q_x)_{x \geqslant 0}$ on $\mathbb R_+$ is called {\it stochastically ordered} if $Q_{x}([z, \infty)) \leqslant Q_y([z, \infty))$ for $0 \leqslant x \leqslant y$ and $z \geqslant 0$. 

We operate on a filtered probability space $(\Omega, \mathcal F, (\mathcal F_t)_{t \geqslant 0}, \mathbb P)$. Consider a Markov process $X = (X(t),\, t \geqslant 0)$ on $\mathbb R_+$ with transition kernel $P^t(x, \cdot)$. Formally, $P^t(x, A) = \mathbb P(X(t) \in A\mid X(0) = x)$. This process $X$ has {\it positivity property} if $P^t(x, B) > 0$ for any $t > 0$, $x \geqslant 0$, and a Borel subset $B \subseteq \mathbb R_+$ of positive Lebesgue measure. For any initial value $X(0) = x$, we can construct a copy of this Markov process. If the initial distribution is $X(0) \sim \rho$, then the distribution of $X(t)$ is written as 
$$
\rho P^t,\quad \mbox{where}\quad \rho P^t(B) := \int_0^{\infty}P^t(x, B)\rho(\mathrm{d}x)\quad \mbox{for}\quad B \subseteq \mathbb R_+.
$$
The transition semigroup $(P^t)_{t \geqslant 0}$ is defined as 
$$
P^tf(x) = \mathbb E\left[f(X(t))\mid X(0) = x\right] = \int_0^{\infty}P^t(x, \mathrm{d}y)f(y).
$$
This process has {\it generator} $\mathcal L$ with domain $\mathcal D(\mathcal L)$:
$$
\mathcal L f(x) = \lim\limits_{t \downarrow 0}\frac1t\left(P^tf(x) - f(x)\right),\quad f \in \mathcal D(\mathcal L).
$$
This Markov process is called {\it stochastically ordered} if for every $t \geqslant 0$, the family of measures $(P^t(x, \cdot))_{x \geqslant 0}$ is stochastically ordered. A probability measure $\pi$ on $\mathbb R_+$ is called  a {\it stationary distribution} for this Markov process if $X(0) \sim \pi$ implies $X(t) \sim \pi$ for all $t > 0$: 
$$
\int_0^{\infty}P^t(x, A)\pi(\mathrm{d}x) = \pi(A)\quad \forall A \subseteq \mathbb R_+.
$$
Take a strictly increasing concave $\varphi : [1, \infty) \to \mathbb R_+$. Define
\begin{equation}
\label{eq:Phi-def}
\Phi(s) := \int_1^s\frac{\mathrm{d}u}{\varphi(u)}.
\end{equation}
Assuming this function is finite for all $s > 1$, it is strictly increasing $\Phi : [1, \infty) \to \mathbb R_+$, thus it has a well-defined inverse $\Psi := \Phi^{-1}$ on $[0, \Phi(\infty))$, with $\Phi(\Psi(v)) \equiv v$ for $v \geqslant 0$. Note that $\Psi(\infty) = \infty$: If $\Psi(\infty) = a < \infty$, then $\Phi(a) = \infty$, which contradicts our assumptions.

A {\it $\varphi$-Lyapunov function} for this process $X$  is a function $V : \mathbb R_+ \to (0, \infty)$ such that the following process is a local $(\mathcal F_t)_{t \geqslant 0}$-supermartingale:
\begin{equation}
\label{eq:phi-Lyapunov}
\varphi(V(X(t\wedge\tau_0))) + \int_0^{t\wedge\tau_0}\mathcal L V(X(s))\,\mathrm{d}s,\quad t \geqslant 0,
\end{equation}
where $\tau_0 := \inf\{t \geqslant 0\mid X(t) = 0\}$ is the hitting time of $0$. 

\begin{remark}
If $V \in \mathcal D(\mathcal L)$, this condition~\eqref{eq:phi-Lyapunov} is equivalent to $\mathcal LV(x) \leqslant -\varphi(V(x))$ for $x > 0$. But it is more convenient for us to write this condition in the form~\eqref{eq:phi-Lyapunov}. We consider functions with $V'(0) \ne 0$, but reflected processes have generator domain $\mathcal D(\mathcal L)$ restricted by $V'(0) = 0$.
\end{remark}

\begin{example}
For $\varphi(s) = ks,\, k > 0$, this becomes {\it modified Lyapunov function} from \cite{Sar16}.
\end{example}

\subsection{Main results} 
Take a function $G(t, u) := \Psi(\Phi(u) + t)$, $u \geqslant 1$, $t \geqslant 0$. 

\begin{theorem} Take a concave increasing function $\varphi : [1, \infty) \to [0, \infty)$ with finite $\Phi$ from~\eqref{eq:Phi-def}. Assume $X$ is a semimartingale and a stochastically ordered Markov process on $\mathbb R_+$. Take a $\varphi$-Lyapunov function $V$ such that for some nondecreasing functions $U, r : \mathbb R_+ \to \mathbb R_+$, 
\begin{equation}
\label{eq:product}
h(t)U(x) \leqslant G(t, V(x)),\quad t, x \geqslant 0.
\end{equation}
Then we have the following results:

\begin{enumerate}[(a)]

\item Take two copies $X_1$ and $X_2$ of $X$ starting from $x_1$ and $x_2$, with $0 \leqslant x_1 \leqslant x_2$. Then
$$
\norm{P^t(x_1, \cdot) - P^t(x_2, \cdot)}_U \leqslant 2h^{-1}(t)V(x_2),\quad t \geqslant 0.
$$

\item Take two copies $X_1$ and $X_2$ of $X$ starting from distributions $\rho_1$ and $\rho_2$ on $\mathbb R_+$. Then we can write $X_i(t) \sim \rho_iP^t$, $i = 1, 2$. Then 
$$
\norm{\rho_1P^t - \rho_2P^t}_U \leqslant 2h^{-1}(t)(\rho_1\vee\rho_2,V),\quad t \geqslant 0.
$$

\item Take a copy $X$ starting from $X(0) \sim \rho$ with $(\rho, V) < \infty$. Assume $X$ has a unique stationary distribution $\pi$ with $(\pi, V) < \infty$, then 
$$
\norm{\rho P^t - \pi}_U \leqslant 2h^{-1}(t)(\pi\vee\rho, V).
$$
\end{enumerate}
\label{thm:main}
\end{theorem}


\begin{remark} In Theorem~\ref{thm:main}, we can let $U := 1$ and $h := \Psi$. Then we get convergence in the total variation norm. The rate of this convergence is given by $1/\Psi(t)$. This is stronger than in \cite{DFG09}, where the rate of convergence for the total variation norm is $1/\varphi(\Psi(t))$. Indeed, in the above example $\varphi(u) := \sqrt{u}$, and so $\varphi(\Psi(t)) = o(\Psi(t))$ as $t \to \infty$. 
\end{remark}

\begin{remark}
\label{rmk:perfect}
Let us discuss these results in exponential ergodicity case, when $\varphi(s) = ks,\, k > 0$. Then $\Phi(s) = k^{-1}\ln(s)$, and $\Psi(v) := e^{kv}$. Thus $G(t, u) = ue^{kt}$, and we can take $U := V$ and $r(t) := e^{kt}$. Here, we have perfect decomposition of $G$ into a product form.
\end{remark}

\begin{example}
Even if $\varphi$ is bounded, we can get nontrivial results. For example, $\varphi(x) = k > 0$ implies $\Phi(x) = k^{-1}(x-1)$ and $\Psi(x) = kx + 1$. Thus $G(t, x) = \Psi(\Phi(x)+t) = x + kt$, and we can let $h(t) := 2kt^{1/2}$, $U(x) := V^{1/2}(x)$, or use more general techniques from the Appendix.
\label{rmk:trivial}
\end{example}


\begin{remark}
The existence and uniqueness of the stationary distribution $\pi$ and the property $(\pi, V) < \infty$ can be obtained from the  Lyapunov conditions in their classic form, as in \cite{DFG09, MT93a, MT93b}. Often we can find (another version of) a Lyapunov function $\hat{V} : \mathbb R_+ \to [1, \infty)$ in $\mathcal D(\mathcal L)$ such that for constants $x_0, b > 0$:
\begin{equation}
\label{eq:old-L}
\mathcal L\hat{V}(x) \leqslant -\varphi(\hat{V}(x)) + b1_{[0, x_0]}(x).
\end{equation}
We also need the following {\it positivity property}: $P^t(x, B) > 0$ for every $t > 0$, $x \geqslant 0$, and every Borel set $B \subseteq \mathbb R_+$ of positive Lebesgue measure. Then the process has a unique stationary distribution $\pi$, and $(\pi, \hat{V}) < \infty$. This follows from \cite[Lemma 2.3, Theorem 2.6]{MyOwn10}; see also \cite{MT93a, MT93b}. In practice, for processes on $\mathbb R_+$, this function $\hat{V}$ can sometimes be constructed as $\hat{V}(x) := V(\psi(x))$, where $\psi : \mathbb R_+ \to \mathbb R_+$ is a nondecreasing function with $\psi(x) = 0$ for $x \leqslant x_1$ and $\psi(x) = x$ for $x \geqslant x_2$, where $0 < x_1 < x_2$; and, finally, $\psi(x) \leqslant x$ for all $x \geqslant 0$. Then $\hat{V}(x) \leqslant V(x) \leqslant C\hat{V}(x)$ for some constant $C > 1$, and $(\pi, \hat{V}) < \infty$ implies $(\pi, V)  < \infty$.
\label{rmk:old-L}
\end{remark}  

\section{Examples}

\subsection{Reflected diffusions} Take functions $g, \sigma : \mathbb R_+ \to \mathbb R$. For a one-dimensional $(\mathcal F_t)_{t \geqslant 0}$-Brownian motion $W = (W(t),\, t \geqslant 0)$, consider a stochastic differential equation (SDE): 
\begin{equation}
\label{eq:SDE}
\mathrm{d}Z(t) = g(Z(t))\,\mathrm{d}t + \sigma(Z(t))\,\mathrm{d}W(t).
\end{equation}

\begin{definition} A solution to the SDE~\eqref{eq:SDE} with reflection on $\mathbb R_+$ starting from $x \geqslant 0$ is defined as an adapted process $Z = (Z(t),\, t \geqslant 0)$ with a.s. continuous trajectories, such that
$$
Z(t) = x + \int_0^tg(Z(s))\,\mathrm{d}s + \int_0^t\sigma(Z(s))\,\mathrm{d}W(s) + \ell(t),\, t \geqslant 0,
$$
where $W$ is an $(\mathcal F_t)_{t \geqslant 0}$-Brownian motion, and $\ell = (\ell(t),\, t \geqslant 0)$ is a continuous nondecreasing process with $\ell(0) = 0$, increasing only when $Z(t) = 0$. It is also called a {\it reflected diffusion} on $\mathbb R_+$ with {\it drift} $g$ and {\it diffusion} $\sigma^2$.
\end{definition}

\noindent This process is well-defined for continuous $g, \sigma$, stochastically ordered and has generator 
$$
\mathcal Lf(x) = g(x)f'(x) + \frac12\sigma^2(x)f''(x),\ f \in C^2(\mathbb R_+),\ f'(0) = 0.
$$
Thus for all functions $f \in C^2(\mathbb R_+)$, even when $f'(0) \ne 0$ and $f$ is not in the domain of the generator, the following process is an $(\mathcal F_t)_{t \geqslant 0}$-supermartingale:
$$
f(Z(t\wedge\tau_0)) - \int_0^{t\wedge\tau_0}\mathcal L f(Z(s))\,\mathrm{d}s,\quad t \geqslant 0.
$$
We can prove that $\mathcal L V(x) \leqslant - \varphi(V(x))$ for $x > 0$, even when $V'(0) \ne 0$, and this would prove that $V$ is a $\varphi$-Lyapunov function. Assume for some constants $a, c > 0$ and $\alpha \in (0, 1)$,
\begin{equation}
\label{eq:drift}
g(x) \leqslant -a(1 + cx)^{\alpha-1},\quad x > 0.
\end{equation}

\begin{example} Take a function $V(x) = 1 + cx$. We get: $V'(x) = c$, $V''(x) = 0$, and thus
\begin{align*}
\mathcal LV(x) \leqslant - ac(1 + cx)^{\alpha-1} = -\varphi(V(x)), \quad \varphi(s) := acs^{1 - \alpha}.
\end{align*}
The function $\varphi$ is increasing and concave. We can compute
\begin{align}
\label{eq:G-SDER}
\begin{split}
\Phi(s) := \int_1^s&\frac{\mathrm{d}y}{\varphi(y)} = \frac1{ac\alpha}\left[s^{\alpha} - 1\right],\quad \Psi(v) := \left[ac\alpha v + 1\right]^{1/\alpha}; \\ G(t, x) &:= \Psi(\Phi(x) + t) = \left[ac\alpha t + x^{\alpha}\right]^{1/\alpha}. 
\end{split}
\end{align}
Thus we can deduce~\eqref{eq:product} from~\eqref{eq:inv-Y} from the Appendix:
\begin{align*}
\Psi(t, V(x)) &=  \left[ac\alpha\cdot  t + a^{\alpha}(1 + cx)^{\alpha}\right]^{1/\alpha} \geqslant
h(t)U(x), \\
h(t) &:= \left[H^{-1}(ac\alpha\cdot  t)\right]^{1/\alpha}, \quad U(x) := \left[K^{-1}(a^{\alpha}(1 + cx)^{\alpha})\right]^{1/\alpha}.
\end{align*}
For the example in~\eqref{eq:p-q}, we get: $H^{-1}(x) = p^{1/p}x^{1/p}$ and $K^{-1}(y) := q^{1/q}y^{1/q}$, thus 
$$
h(t) := (acp\alpha)^{1/\alpha p}t^{1/(\alpha p)},\quad 
U(x) := q^{1/(q\alpha)}a^{1/\alpha}(1 + cx)^{1/q}.
$$
If $p$ is larger, then $q$ is smaller. This illustrates the general principle: Taking a stronger $\norm{\cdot}_U$, for a larger function $U$, leads to slower convergence rate $h$, and vice versa. 
\label{exm:linear}
\end{example}

\begin{example} Now assume the condition~\eqref{eq:drift} holds, and moreover $\sigma(x) \equiv \sigma = \mathrm{const}$. Take a function $V(x) := (1 + \lambda x)^{\beta}$ for $\beta > 1$ and $\lambda \geqslant c$ to be determined later. Then we have: 
\begin{align*}
V'(x) &= \beta\lambda\cdot(1 + \lambda x)^{\beta - 1},\quad V''(x) = \lambda^2\beta(\beta - 1)(1 + \lambda x)^{\beta - 2},\\
\mathcal LV(x) \leqslant -a\beta&\lambda(1 + cx)^{\alpha-1}(1 + \lambda x)^{\beta - 1}  + \frac12\lambda^2\sigma^2\beta(\beta - 1)(1 + \lambda x)^{\beta - 2} \leqslant -A(\lambda)(1 + \lambda x)^{\beta - \alpha},\\ & A(\lambda) := a\beta\lambda -  \frac12\sigma^2\beta(\beta - 1)\lambda^2,\quad x > 0.
\end{align*}
Thus we can take $\varphi(s) := A(\lambda)s^{1 - \alpha/\beta}$. Since $\alpha < 1 < \beta$, then $0 < 1 - \alpha/\beta < 1$ and this is a concave increasing function. This is only true if we can make sure that $A(\lambda) > 0$. We need  
\begin{equation}
\label{eq:lambda}
\lambda < \frac{2a}{\sigma^2(\beta - 1)}.
\end{equation}
But on the other hand, we need $\lambda \geqslant c$. We got a prerequisite inequality:
\begin{equation}
\label{eq:beta}
\frac{2a}{\sigma^2(\beta - 1)} > c\, \Leftrightarrow\, \beta < 1 + \frac{\sigma^2}{2ac}.
\end{equation}
If we satisfy conditions~\eqref{eq:lambda} and~\eqref{eq:beta}, then we can constuct this Lyapunov function. It gives us better rates of convergence and stronger $U$-norm $\norm{\cdot}_U$. But we did this under~\eqref{eq:drift} combined with constant diffusion coefficient. If we have similar estimates for a non-constant diffusion coefficient, we need to modify our Lyapunov function. 
\end{example}

\begin{remark} Let us follow Remark~\ref{rmk:old-L} to show for both examples that the stationary distribution $\pi$ exists and is unique, and $(\pi, V) < \infty$. For each case, taking $\hat{V}(x) := V(\psi(x))$ as at the end of Section 2, we get a function $\hat{V}$ such that $\hat{V}'(0) = 0$ and the condition~\eqref{eq:old-L} holds. Since the process $X$ has the positivity property $P^t(x, B) > 0$, it has a unique stationary distribution $\pi$ which satisfies $(\pi, \hat{V}) < \infty$. Since $\hat{V}(x) = V(x)$ for large enough $x$, we can automatically conclude that $(\pi, V) < \infty$. 
\end{remark}

\subsection{Reflected jump-diffusions} In addiition to the above notation, take a family $(\nu_x)_{x \geqslant 0}$ of finite Borel measures on $\mathbb R_+$ with $\nu_x(\mathbb R_+) = M < \infty$ for all $x \geqslant 0$. Then we can augment the above reflected diffusion with these jumps: They occur with intensity $M$, and the destination of a jump from $x \geqslant 0$ is distributed as $M^{-1}\nu_x(\cdot)$. As long as the family $(\nu_x)_{x \geqslant 0}$ is weakly continuous: $\nu_y \to \nu_x$ weakly as $y \to x$, and $g, \sigma$ satisfy the same assumptions as in the previous subsection, this process can be constructed by \textit{piecing out}, see \cite{Sar16}, and it exists in the weak sense and is unique in law. This is a Markov process with the following generator, for $f \in C^2(\mathbb R_+)$ with $f'(0) = 0$:
\begin{equation}
\label{eq:L-jumps}
\mathcal Lf(x) =  g(x)f'(x) + \frac12\sigma^2(x)f''(x) + \int_0^{\infty}[f(y) - f(x)]\,\nu_x(\mathrm{d}y).
\end{equation}
It is shown in \cite{Sar16} that if this family of measures is stochastically ordered, then the process $Z$ is also stochastically ordered. Take a Lyapunov function $V(x) = 1 + cx$ for $c > 0$. Applying the generator from~\eqref{eq:L-jumps} even though $V'(0) \ne 0$, we get:
$$
\mathcal V(x) = cm(x),\quad m(x) := g(x) + \int_0^{\infty}(y-x)\,\nu_x(\mathrm{d}y).
$$
One can think of $m(x)$ as ``average drift'' at the point $x \geqslant 0$ which is the usual drift $g(x)$ combined with the ``implied drift'' created by jumps. This is the ``velocity'' with which the process ``wants to move'' to the right while at location $x$. If $m(x)$ instead of $g(x)$ satisfies~\eqref{eq:drift}, then we can make the same conclusions as for the reflected diffusion above. We need to show that the stationary distribution $\pi$ exists, is unique, and satisfies $(\pi, V) < \infty$, as in Remark~\ref{rmk:old-L}. 

Assume that there exist $0 < c_1 < c_2$ such that for $x \geqslant c_2$, $\nu_x([0, c_1]) = 0$. Take $\psi(x) = 0$ for $x \leqslant c_1/2$ and $\psi(x) = x$ for $x \geqslant c_1$. Then for $x \geqslant c_2$ we have:
\begin{align*}
\int_0^{\infty}&(\hat{V}(y) - \hat{V}(x))\,\nu_x(\mathrm{d}y)  = \int_0^{\infty}(V(\psi(y)) - V(\psi(x)))\,\nu_x(\mathrm{d}y) \\ & = \int_{c_1}^{\infty}(V(\psi(y)) - V(\psi(x)))\,\nu_x(\mathrm{d}y) = \int_{c_1}^{\infty}(V(y) - V(x))\,\nu_x(\mathrm{d}y) \\ & = \int_0^{\infty}(V(y) - V(x))\,\nu_x(\mathrm{d}y);\\
g(x)&\hat{V}'(x) + \frac12\sigma^2(x)\hat{V}''(x) = g(x)V'(x) + \frac12\sigma^2(x)V''(x);
\end{align*}
and thus the generator formally applied to $V$ will be the same as if applied to $\hat{V} \in \mathcal D(\mathcal L)$:
$$
\mathcal L\hat{V}(x) = \mathcal LV(x),\quad x \geqslant c_2.
$$
If $\mathcal L\hat{V}(x)$ is bounded on $[0, c_2]$ (it is easy to check in practice) then we have the condition~\eqref{eq:old-L}. Together with positivity property, proved in \cite{Sar16}, this implies existence and uniqueness of the stationary distribution $\pi$ and $(\pi, V) < \infty$. 

\begin{example}
Try the following process: $g(x) = -3(x+1)^{-0.5}$, $M = 2$ and $\nu_x$ is defined as
$$
\nu_x(\mathrm{d}y) = \lambda(x)\exp\left[\lambda(x)(x-y)\right]1_{\{y > x\}}\,\mathrm{d}y,\quad \lambda(x) := (x+1)^{0.5}.
$$ 
In other words, the process jumps with constant intensity $2$ to the right, and the displacement is distributed as an exponential random variable with rate $\lambda(x)$ (and therefore mean $\lambda^{-1}(x)$). Then the average rate of displacement from jumps is $2\lambda^{-1}(x)$. Thus
$$
m(x) = -3(x+1)^{-0.5} + 2(x+1)^{-0.5} = -(x+1)^{-0.5},\quad x \geqslant 0,
$$
This gives us subexponential convergence, as in Example~\ref{exm:linear} above.
\end{example}

We could also weaken the assumption of finite first moment for $\nu_x$ Instead of stating a general result, we consider a particular case. Assume that there exists a finite measure $\mu$ on the positive half-line $\mathbb R_+$ such that for all $B \subseteq \mathbb R_+$, 
$$
\nu_x(x + B) = \mu(B),\, \nu_x[0, x) = 0,\, x \in \mathbb R_+.
$$
We fix a $\beta \in (0, 1)$ and assume
$$
m_{\beta} := \int_0^{\infty}x^{\beta}\mu(\mathrm{d}x) < \infty.
$$
In addition, assume there exist constants $C, x_0 > 0$ such that $g(x) \leqslant -Cx^{1 - \beta}$ for $x \geqslant x_0$. Then we can take a Lyapunov function $V(x) = 1 + x^{\beta}$. We get: For $x \geqslant x_0$, 
\begin{align*}
\mathcal L V(x) &= g(x)V'(x) + \frac{\sigma^2(x)}2V''(x) + \int_0^{\infty}[V(x+y) - V(x)]\mu(\mathrm{d}y) \\ & = g(x)\beta x^{\beta - 1} + \frac{\sigma^2(x)}2\beta(\beta - 1)x^{\beta - 2} + \int_0^{\infty}\left[(x+y)^{\beta} - x^{\beta}\right]\mu(\mathrm{d}y) \\ & \leqslant -C + \int_0^{\infty}y^{\beta}\mu(\mathrm{d}y) = -C + m_{\beta}.
\end{align*}
The last inequality follows from the estimate $(x+y)^{\beta} \leqslant x^{\beta} + y^{\beta}$ for $x, y \geqslant 0$. If $m_{\beta} < C$, then Theorem~\ref{thm:main} holds with $\varphi$ as in Remark~\ref{rmk:trivial}.

\subsection{Reflected L\'evy processes} Any L\'evy process on the real line can be decomposed into the sum of two components: a Brownian motion with constant drift and diffusion (the continuous component), and a pure jump L\'evy process $\mathcal J$. In this article, we assume this pure jump process is nondecreasing. Its jumps are governed by a {\it spectral measure} $\mu$. If the spectral measure is finite, then there are a.s. finitely many jumps (a Poisson number) during any time interval $[0, t]$, and $\mathcal J$ is, in fact, a compound Poisson process. However, if $\mu(\mathbb R_+) = \infty$, then this process $\mathcal J$ makes infinitely many jumps during any time interval $[0, t]$. This $\mu$ is a $\sigma$-finite Borel measure on $\mathbb R_+$. This measure must satisfy 
\begin{equation}
\label{eq:finite}
\int_0^{\infty}(1\wedge x)\,\mu(\mathrm{d}x) < \infty.
\end{equation}
We take a reflected version $X$ of $L$ on the half-line using the Skorohod reflection mapping
$$
X(t) := L(t) + \sup\limits_{0 \leqslant s \leqslant t}\left[\max(-L(s), 0)\right].
$$
The resulting process will have values in $\mathbb R_+$ and behave like a L\'evy process $L$ as long as it is strictly inside $\mathbb R_+$. When it hits $0$, it is reflected back. If $\mu$ is finite, we are back in the case of reflected jump-diffusions, discussed in the previous subsection. In this subsection, we are mostly interested in the case when $\mu(\mathbb R_+) = \infty$. This process is defined by a triple $(g, \sigma, \mu)$ where $g \in \mathbb R_+$ is a drift, $\sigma > 0$ is a diffusion coefficient, and it can be decomposed as
\begin{equation}
\label{eq:L-decomp}
L(t) = L(0) + gt + \sigma W(t) + \mathcal J(t),\quad t \geqslant 0,
\end{equation}
where $W$ is a Brownian motion. The generator of $L$ is given by
\begin{equation}
\label{eq:Levy-right}
\mathcal L f(x) = gf'(x) + \frac12\sigma^2f''(x) + \int_0^{\infty}\left[f(x + z) - f(x)\right]\,\mu(\mathrm{d}z).
\end{equation}
Impose a condition which is a bit stronger than~\eqref{eq:finite}, that the first moment is finite:
\begin{equation}
\label{eq:moment1}
m_1 := \int_0^{\infty}z\,\mu(\mathrm{d}z) < \infty.
\end{equation}
The reflected version has the same generator~\eqref{eq:Levy-right}, but for functions $f \in C^2(\mathbb R_+)$ with $f'(0) = 0$. Formally apply the generator to the function $V(x) := 1 + x$, although $V'(0) \ne 0$. Then $\mathcal L V(x) = g + m_1$. If $g < -m_1$, then this process is ergodic. We prove this by taking $\hat{V}(x) := V(\psi(x))$ as in Remark~\ref{rmk:old-L}. Then we can apply Theorem~\ref{thm:main} using Remark~\ref{rmk:trivial}. Similarly to the previous subsection, we get  Moreover, assume another condition stronger than~\eqref{eq:moment1}:
\begin{equation}
\label{eq:moment-exp}
\int_1^{\infty}e^{\lambda_0x}\,\mu(\mathrm{d}x) < \infty.
\end{equation}
Take $V(x) := e^{\lambda x}$ for $\lambda \in (0, \lambda_0)$. Then 
\begin{align*}
\mathcal LV(x) &= V(x)k(\lambda),\quad
k(\lambda)  := \lambda g + \frac{\sigma^2\lambda^2}2 + \int_0^{\infty}\left[e^{\lambda z} - 1\right]\,\mu(\mathrm{d}z).
\end{align*}
And the integral in the right-hand side is finite because of~\eqref{eq:moment1} and~\eqref{eq:moment-exp}. Similarly to \cite{Sar16}, we prove in Lemma~\ref{lemma:lambda} that there exists a $\lambda > 0$ such that $k(\lambda) < 0$. Reflected L\'evy process also has positivity property, by Lemma~\ref{lemma:positive}. Thus by using Remark~\ref{rmk:old-L} we get that there exists a unique stationary distribution $\pi$, which satisfies $(\pi, V) < \infty$, and for any such $\lambda$ we get:
$$
\norm{P^t(x, \cdot) - \pi(\cdot)}_{V} \leqslant (V(x) + (\pi, V))e^{k(\lambda)t},\quad x, t \geqslant 0.
$$
This is the same exponential rate of convergence as in \cite{Sar16} for reflected L\'evy processes with finite spectral measure (which is a compound Poisson process plus a Brownian motion).  

\begin{lemma}Under conditions~\eqref{eq:moment1} and~\eqref{eq:moment-exp}, if $g < -m_1$, then $k(\lambda) < 0$ for some $\lambda > 0$.
\label{lemma:lambda}
\end{lemma}

\begin{proof}
This follows from differentiability under the integral with respect to $\lambda$ at $\lambda = 0$, which we shall prove now. For all $\lambda \in (0, \lambda_0)$, there exist constants $C_1, C_2 > 0$ such that for all $z \geqslant 0$, we get: $ze^{\lambda z} \leqslant C_1z + C_2e^{\lambda_0z}$. They can be taken independent of 
$\lambda$. From this elementary estimate combined with~\eqref{eq:moment1} and~\eqref{eq:moment-exp}, we get:
$$
\int_0^{\infty}\left[ze^{\lambda z}\right]\,\mu(\mathrm{d}z) < \infty.
$$
From here, we can deduce by Lebesgue dominated convergence theorem that 
$$
\frac{\mathrm{d}}{\mathrm{d}\lambda}\int_0^{\infty}\left[e^{\lambda z} - 1\right]\,\mu(\mathrm{d}z) = \int_0^{\infty}\frac{\partial}{\partial\lambda}\left[e^{\lambda z} - 1\right]\,\mu(\mathrm{d}z) = \int_0^{\infty}\left[ze^{\lambda z}\right]\,\mu(\mathrm{d}z).
$$
Thus, taking derivative of $k(\lambda)$ and letting $\lambda = 0$, we get: 
$$
k'(0) = g + \int_0^{\infty}z\,\mu(\mathrm{d}z) = g + m_1 < 0.
$$
This, together with previous computations, completes the proof.
\end{proof}

\begin{lemma} For any $t > 0$, $x \geqslant 0$, and a set $B \subseteq \mathbb R_+$ of positive Lebesgue measure, the transition kernel of the reflected process satisfies $P^t(x, B) > 0$.
\label{lemma:positive}
\end{lemma}

\begin{proof}
Assume $x > 0$. Without loss of generality, we can assume $B \subseteq [\delta, \infty)$ for some $\delta > 0$. Then with probability $p > 0$, this process never hits $0$ until time $t$. In this case, it behaves as a non-reflected L\'evy process, which has positivity property $Q^t(x, B) > 0$. Thus $P^t(x, B) \geqslant pQ^t(x, B) > 0$. Next, if $x = 0$, then $P^{t/2}(0, (0, \infty)) > 0$, and thus $P^t(x, B) \geqslant \int_0^{\infty}P^{t/2}(y, B)P^{t/2}(0, \mathrm{d}y) > 0$ (the integral of a positive function over a positive measure is positive). Thus we reduced the case $x = 0$ to the case $x > 0$.
\end{proof}

\section{Proof of Theorem~\ref{thm:main}}

\subsection{Overview of the proof} The main idea is similar to \cite[Theorem 4.1, Theorem 5.2]{Sar16}. It is known from \cite[Theorem 5]{Order} that if $X$ has trajectories which are a.s. right-continuous with left limits, then being stochastically ordered is equivalent to the following statement: For every $0 \leqslant x_1 \leqslant x_2$, there exists a {\it coupling} of two versions $X_1$ and $X_2$ of $X$, starting from $X_1(0) = x_1$ and $X_2(0) = x_2$: A probability space and copies $X_1$ and $X_2$ defined on this space such that $X_1(t) \leqslant X_2(t)$ for all $t \geqslant 0$ a.s. Let us define a coupling time $\tau_0 := \inf\{t \geqslant 0 \mid X_1(t) = X_2(t)\}$. We can assume $X_1(t) = X_2(t)$ for $t > \tau_0$, so after this coupling time, the processes coincide. For any function $g : \mathbb R_+ \to \mathbb R$ with $|g| \leqslant U$, we get (the last inequality from nondecreasing $U$):
\begin{align}
\label{eq:first-est}
\begin{split}
h(t)&\left|\mathbb E[g(X_1(t))] - \mathbb E[g(X_2(t))]\right| = h(t)\left|\mathbb E[g(X_1(t))]1_{\{\tau_0 > t\}} - \mathbb E[g(X_2(t))1_{\{\tau_0 > t\}}]\right| \\ & \leqslant h(t)\left(\mathbb E\left[U(X_1(t))1_{\{\tau_0 > t\}}\right] + \mathbb E\left[U(X_2(t))1_{\{\tau_0 > t\}}\right]\right) \leqslant 2\mathbb E\left[h(t)U(X_2(t))1_{\{\tau_0 > t\}}\right]. 
\end{split}
\end{align} 
Now we define $\tau := \inf\{t \geqslant 0\mid X_2(t) = 0\}$. Then $X_1(\tau) \leqslant X_2(\tau) = 0$ and thus $X_1(\tau) = 0$. Therefore, $\tau \geqslant \tau_0$ a.s. Combining this observation with~\eqref{eq:first-est}, we get:
\begin{equation}
\label{eq:modified-first-est}
h(t)\left|\mathbb E[g(X_1(t))] - \mathbb E[g(X_2(t))]\right| \leqslant 2\mathbb E\left[h(t)U(X_2(t))1_{\{\tau > t\}}\right]. 
\end{equation}

\begin{lemma} The function $G$ satisfies the boundary conditions 
$$
G(t, 1) = \Psi(t),\, t \geqslant 0;\quad G(0, u) = u,\, u \geqslant 1,
$$
and for $t \geqslant 0,\, u \geqslant 1$ the following equation and inequalities hold:
\begin{equation}
\label{eq:results}
\frac{\partial G}{\partial t}(t, u) = \varphi(u)\frac{\partial G}{\partial u}(t, u),\quad \frac{\partial G}{\partial u}(t, u) \geqslant 0,\quad \frac{\partial^2G}{\partial u^2}(t, u) \leqslant 0.
\end{equation}
\label{lemma:G}
\end{lemma}

This technical lemma was partially proved in \cite{DFG09}, but we want to collect these results and give (a straightforward) proof for the sake of completeness. The key part of the proof of Theorem~\ref{thm:main} is the following lemma. Take a copy of $X$ starting from $X(0) = x_0 \geqslant 0$.

\begin{lemma} The following process $K = (K(t),\, t \geqslant 0)$ is a local $(\mathcal F_t)_{t \geqslant 0}$-supermartingale:
$$
K(t) := G(t\wedge\tau,\, V(X(t\wedge\tau))),\, t \geqslant 0.
$$
\label{lemma:main}
\end{lemma}

Assume Lemma~\ref{lemma:main} is already proved. Let us show that $\tau < \infty$ a.s. Since $G$ is nondecreasing with respect to $u$, we get: $G(t, u) \geqslant G(t, 1) = \Psi(t)$. Letting $n = 1, 2, \ldots$ we get: 
\begin{equation}
\label{eq:G-comp}
G(0, V(X(0))) = \mathbb E K(0) \geqslant \mathbb E K(n) \geqslant \mathbb E\left[G(n\wedge\tau, V(X(n\wedge\tau)))\right] \geqslant \mathbb E\Psi(n\wedge\tau).
\end{equation}
Assume the converse: $\tau = \infty$ with positive probability. Then 
\begin{equation}
\label{eq:Psi-comp}
\mathbb E\Psi(n\wedge\tau) \geqslant \Psi(n)\mathbb P(\tau = \infty).
\end{equation}
But $\Psi(\infty) = \infty$, as stated in subsection 2.1. Letting $n \to \infty$, we compare~\eqref{eq:G-comp} and~\eqref{eq:Psi-comp} and arrive at a contradiction. This proves that $\tau < \infty$ a.s. Since $K(t) \geqslant 0$ for all $t \geqslant 0$, then by Fatou's lemma we can remove the word ``local'' from Lemma~\ref{lemma:main}. From~\eqref{eq:product} we get:
\begin{align}
\label{eq:final-est}
\begin{split}
h(t)&\mathbb E\left[U(X_2(t))1_{\{\tau > t\}}\right] = \mathbb E\left[h(t\wedge\tau)U(X_2(t\wedge\tau))1_{\{\tau > t\}}\right]\\ & \leqslant \mathbb E\left[G(t\wedge\tau, V(X_2(t\wedge\tau)))1_{\{\tau > t\}}\right] \leqslant \mathbb E\left[G(t\wedge\tau, V(X_2(t\wedge\tau)))\right] \\ & = \mathbb E K(t) \leqslant \mathbb E K(0) = G(0, V(x_2)) = V(x_2).
\end{split}
\end{align}
Combining~\eqref{eq:first-est} with~\eqref{eq:final-est}, we get: 
$$
h(t)\left|\mathbb E[g(X_1(t))] - \mathbb E[g(X_2(t))]\right| \leqslant 2V(x_2).
$$
Dividing by $h(t)$ and taking the supremum over all $g : \mathbb R_+ \to \mathbb R$ with $|g| \leqslant U$, we complete the proof of (a). The proof of (b) is done similarly, and (c) follows from (b) and the observation that $\pi P^t = \pi$ for a stationary distribution $\pi$ and any $t \geqslant 0$. 

\subsection{Proof of Lemma~\ref{lemma:G}}  The boundary conditions are easy to show: $G(0, u) = \Psi(\Phi(u)) = u$ for $u \geqslant 1$, and $G(t, 1) = \Psi(\Phi(1) + t) = \Psi(t)$ for $t \geqslant 0$. Next, the function $\varphi : [1, \infty) \to \mathbb R_+$ is increasing. Thus the function $\Phi : [1, \infty) \to \mathbb R_+$ satisfies $\Phi(1) = 0$, increasing, and concave. Next, $\Psi : \mathbb R_+ \to [1, \infty)$ satisfies $\Psi(0) = 1$, increasing, and convex.  By the chain rule, $\Psi'(v) = \left[\Phi'(\Psi(v))\right]^{-1} = \varphi(\Psi(v))$. Let us do computation of derivatives of $G$:
\begin{align}
\label{eq:O}
\begin{split}
\frac{\partial G}{\partial t}(t, u) &= \Psi'(\Phi(u) + t) = \varphi(\Psi(\Phi(u) + t)) = \varphi(G(t, u));\\
\frac{\partial G}{\partial u}(t, u) &= \Psi'(\Phi(u)+t)\cdot\Phi'(u) = \frac{\varphi(\Psi(\Phi(u)+t))}{\varphi(u)} = \frac{\varphi(G(t, u))}{\varphi(u)} \geqslant 0.
\end{split}
\end{align}
These two equations from~\eqref{eq:O} prove the equality and the first inequality in~\eqref{eq:results}. Finally, we need to prove the second inequality in~\eqref{eq:results}, that is, concavity of $G$ with respect to $u$:
\begin{equation}
\label{eq:A}
\frac{\partial^2 G}{\partial u^2}(t, u) = \frac{\partial}{\partial u}\left( \frac{\varphi(G(t, u))}{\varphi(u)}\right) = 
\frac1{\varphi^2(u)}\left[\varphi'(G(t, u))\frac{\partial G}{\partial u}(t, u)\varphi(u) - \varphi(G(t, u))\varphi'(u)\right].
\end{equation}
Next, using the second equation in~\eqref{eq:O}, we get:
\begin{equation}
\label{eq:B}
\varphi'(G(t, u))\frac{\partial G}{\partial u}(t, u)\varphi(u) = \varphi'(G(t, u))\varphi(G(t, u)).
\end{equation}
Finally, $\varphi'$ is nonincreasing. Since $\Psi$ is nondecreasing, $G(t, u) = \Psi(\Phi(u)+t) \geqslant \Psi(\Phi(u)) = u$ for all $u \geqslant 1$ and $t \geqslant 0$. Thus $\varphi'(G(t, u)) \leqslant \varphi'(u)$. Multiplying this inequality by $\varphi(G(t, u)) \geqslant 0$ and using~\eqref{eq:A} and~\eqref{eq:B}, we complete the proof.

\subsection{Proof of Lemma~\ref{lemma:main}} We know that the following process is an $(\mathcal F_t)_{t \geqslant 0}$-supermartingale:
\begin{equation}
\label{eq:N}
N(t) = V(X(t\wedge\tau)) + \int_0^{t\wedge\tau}\varphi(V(X(s)))\,\mathrm{d}s,\, t \geqslant 0.
\end{equation}
When we write in differential notation, we let $df(t) \leqslant 0$ if $f$ is a nonincreasing function, and $\mathrm{d}f(t) \leqslant \mathrm{d}g(t)$ if $g - f$ is a nondecreasing function. For $t < \tau$, since $V$ is smooth and $X$ is a semimartingale, then $V(X(\cdot))$ is also a semimartingale. Thus we can apply It\^o's formula:
\begin{align}
\label{eq:K0}
\begin{split}
&\mathrm{d}K(t)  = \mathrm{d}G(t, V(X(t))) \\ & = \frac{\partial G}{\partial t}(t, V(X(t)))\,\mathrm{d}t + \frac{\partial G}{\partial u}(t, V(X(t)))\,\mathrm{d}V(X(t)) + \frac12\frac{\partial^2G}{\partial u^2}(t, V(X(t)))\,\mathrm{d}\langle V(X)\rangle_t.
\end{split}
\end{align}
From concavity of $G$ with respect to $u$ found in Lemma~\ref{lemma:G}, and from~\eqref{eq:K0}, we get:
\begin{equation}
\label{eq:K}
\mathrm{d}K(t) \leqslant \frac{\partial G}{\partial t}(t, V(X(t)))\,\mathrm{d}t + \frac{\partial G}{\partial u}(t, V(X(t)))\,\mathrm{d}V(X(t)).
\end{equation}
On the other hand, from~\eqref{eq:N} we get: $\mathrm{d}N(t) \leqslant \mathrm{d}V(X(t)) + \varphi(V(X(t)))\,\mathrm{d}t + \mathrm{d}M(t)$ for some local $(\mathcal F_t)_{t \geqslant 0}$-martingale $M = (M(t),\, t \geqslant 0)$. Thus for $t < \tau$ we have:
$\mathrm{d}V(X(t)) = \mathrm{d}N(t) - \mathrm{d}M(t) - \varphi(V(X(t)))\,\mathrm{d}t$. Next, from above calculations in~\eqref{eq:K} and~\eqref{eq:N}, we get: 
\begin{align}
\label{eq:K-N}
\begin{split}
\mathrm{d}K(t) & \leqslant \frac{\partial G}{\partial t}(t, V(X(t)))\,\mathrm{d}t 
 + \frac{\partial G}{\partial u}(t, V(X(t)))\,\mathrm{d}N(t) \\ & -  \frac{\partial G}{\partial u}(t, V(X(t)))\,\mathrm{d}M(t) -  \frac{\partial G}{\partial u}(t, V(X(t)))\,\varphi(V(X(t)))\,\mathrm{d}t.
\end{split}
\end{align}
From computations of derivatives of $G$ we get:
\begin{equation}
\label{eq:derivatives}
\frac{\partial G}{\partial t}(t, u) = \varphi(u)\frac{\partial G}{\partial u}(t, u).
\end{equation}
Plugging~\eqref{eq:derivatives} into~\eqref{eq:K-N}, we get:
$\mathrm{d}K(t) \leqslant \frac{\partial G}{\partial u}(t, V(X(t)))\,(\mathrm{d}N(t) - \mathrm{d}M(t))$. But $N$ is an $(\mathcal F_t)_{t \geqslant 0}$-supermartingale, and $M$ is a local $(\mathcal F_t)_{t \geqslant 0}$-martingale. Thus $N - M$ is also an $(\mathcal F_t)_{t \geqslant 0}$-supermartingale, and the same can be said about $K$, since $\frac{\partial G}{\partial u} \geqslant 0$. This completes the proof. 

\section{Appendix} 

Below is a method  to find $h$ and $U$ such that~\eqref{eq:product} holds. Take any {\it Young pair} $(H, K)$ of strictly increasing functions $\mathbb R_+ \to \mathbb R_+$ such that 
\begin{equation}
\label{eq:Young}
H(0) = K(0) = 0,\quad H(\infty) = K(\infty) = \infty,\quad xy \leqslant H(x) + K(y),\quad x, y \geqslant 0.
\end{equation}
The following example follows from Young's inequality:
\begin{equation}
\label{eq:p-q}
H(x) = \frac{x^{p}}p,\quad K(y) := \frac{y^q}{q},\quad p,\, q > 0,\quad \frac1p + \frac1q = 1.
\end{equation}
Assume~\eqref{eq:Young} holds. Since $H$ and $K$ are strictiy increasing, define their inverses $H^{-1}, K^{-1} : \mathbb R_+ \to \mathbb R_+$ such that $H(H^{-1}(x)) = x$ and $K(K^{-1}(y)) = y$ for all $x, y \geqslant 0$. Then
\begin{equation}
\label{eq:inv-Y}
H^{-1}(x)K^{-1}(y) \leqslant x + y,\, x, y \geqslant 0.
\end{equation}
For example, inverses of~\eqref{eq:p-q} are given by $H^{-1}(x) := (px)^{1/p}$, $K^{-1}(y) := (qy)^{1/q}$.


\begin{thebibliography}{99}

\bibitem{Asmussen} \textsc{Soren Asmussen} (1998). Subexponential Asymptotics for Stochastic Processes: Extremal Behavior, Stationary Distributions and First Passage Probabilities. \textit{Annals of Applied Probability} \textbf{8}, 354--374.

\bibitem{AB} \textsc{Rami Atar, Amarjit Budhiraja} (2002). Stability Properties of Constrained Jump-Diffusion Processes. \textit{Electronic Journal of Probability} \textbf{7}, no. 22.

\bibitem{M1} \textsc{Boris Baeumer, Mih\'aly Kov\'acs, Mark M. Meerschaert, Ren\'e L. Schilling, Peter Straka} (2016). Reflected Spectrally Negative Stable Processes and Their Governing Equations. \textit{Transactions of the American Mathematical Society} \textbf{368}, 227--248.

\bibitem{M2} \textsc{Boris Baeumer, Mih\'aly Kov\'acs, Mark M. Meerschaert, Harish Sankaranarayanan} Boundary Conditions for Fractional Diffusion, \textit{Journal of Computational and Applied Mathematics}, \textbf{336}, 408--424. 

\bibitem{BCG08} \textsc{Dominique Bakry, Patrick Cattiaux, Arnaud Guillin} (2008). Rate of Convergence of Ergodic Continuous Markov Chains: Lyapunov vs Poincar\'e. \textit{Journal of Functional Analysis} \textbf{254}, 727--754. 

\bibitem{Ben} \textsc{Jonathan Ben-Artzi, Amit Einav} (2020). Weak Poincar\'e Inequalities in the Absence of Spectral Gaps. \textit{Annales Henri Poincar\'e} \textbf{21}, 359--375.

\bibitem{Qnew} \textsc{Anton Braverman, Jim G. Dai, Masakiyo Miyazawa} (2017). Heavy Traffic Approximation for the Stationary Distribution of a Generalized Jackson Network: The BAR Approach. \textit{Stochastic Systems} \textbf{7}, 143--196.

\bibitem{Butkovsky} \textsc{Oleg Butkovsky} (2014). Subgeometic Rate of Convergence of Markov Processes in the Wasserstein Metric. \textit{Annals of Applied Probability} \textbf{24} (2), 526--552. 

\bibitem{Davies} \textsc{P. Laurie Davies} (1986). Rates of Convergence to the Stationary Distribution for $k$-Dimensional Diffusion Processes. \textit{Journal of Applied Probability} \textbf{23}, 370--384. 

\bibitem{Levy} \textsc{Krzysztof Debicki, Kamil Marcin Kosi\'nski, Michel Mandjes} (2012). On the Infimum Attained by a Reflected L\'evy Process. \textit{Queueing Systems} \textbf{70}, 23--35. 

\bibitem{DFG09} \textsc{Randal Douc, Gersende Fort, Arnaud Guillin} (2009). Subgeometric Rates of Convergence of $f$-Ergodic Strong Markov Processes. \textit{Stochastic Processes and Their Applications} \textbf{119}, 897--923.

\bibitem{DFMS04} \textsc{Randal Douc, Gersende Fort, Eric Moulines, Philippe Soulier} (2004). Practical Drift Conditions for Subgeometric Rates of Convergence. \textit{Annals of Applied Probability} \textbf{14}, 1353--1377.

\bibitem{DMT95} \textsc{Douglas Down, Sean P. Meyn, Richard L. Tweedie} (1995). Exponential and Uniform Ergodicity of Markov Processes. \textit{Annals of Probability} \textbf{23}, 1671--1691. 

\bibitem{Survey} \textsc{E. Robert Fernholz, Ioannis Karatzas} (2009). Stochastic Portfolio Theory: an Overview. \textit{Handbook of Numerical Analysis} \textbf{15}, 89--167.

\bibitem{FM} \textsc{Gersende Fort, Eric Moulines} $V$-Subgeometric Ergodicity for a Hastings-Metropolis Algorithm. \textit{Statistics and Probability Letters} \textbf{49}, 401--410.

\bibitem{FR05} \textsc{Gersende Fort, Gareth O. Roberts} (2005). Subgeometric Ergodicity of Strong Markov Processes. \textit{Annals of Applied Probability} \textbf{15}, 1565--1589.

\bibitem{Risk} \textsc{Pierre-Olivier Goffard, Andrey Sarantsev} (2019). Exponential Convergence Rate of Ruin Probabilities for Level-Dependent Levy-Driven Risk Processes. \textit{Journal of Applied Probability} \textbf{56}, 1244--1268.

\bibitem{Degenerate} \textsc{Martin Grothaus, Feng-Yu Wang} (2019). Weak Poincar\'e Inequalities for Convergence Rate of Degenerate Diffusion Processes. \textit{Annals of Probability} \textbf{47} (5), 2930--2952.

\bibitem{SPT} \textsc{Tomoyuki Ichiba, Soumik Pal, Mykhaylo Shkolnikov} (2013). Convergence Rates for Rank-Based Models with Applications to Portfolio Theory. \textit{Probability Theory \& Related Fields} \textbf{156}, 415--448. 

\bibitem{Walsh} \textsc{Tomoyuki Ichiba, Andrey Sarantsev} (2019). Convergence and Stationary Distributions for Walsh Diffusions. \textit{Bernoulli} \textbf{25}, 2439--2478. 

\bibitem{Q} \textsc{Offer Kella, Ward Whitt} (1990). Diffusion Approximations for Queues with Server Vacations. \textit{Advances in Applied Probability} \textbf{22}, 706--729.

\bibitem{LedouxBook} \textsc{Michel Ledoux} (2005). \textit{The Concentration of Measure Phenomenon}. Mathematical Surveys \& Monographs \textbf{89}. American Mathematical Society.

\bibitem{LMT96} \textsc{Robert B. Lund, Sean P. Meyn, Richard L. Tweedie} (1996). Computable Exponential Convergence Rates for Stochastically Ordered Markov Processes. \textit{Annals of Applied Probability} \textbf{6}, 218--237. 

\bibitem{LT96} \textsc{Robert B. Lund, Richard L. Tweedie} (1996). Geometric Convergence Rates for Stochastically Ordered Markov Chains. \textit{Mathematics of Operations Research} \textbf{21}, 182--194.

\bibitem{Order} \textsc{Teturo Kamae, Ulrich Krengel, George L. O'Brien} 
(1977). Stochastic Inequalities on Partially Ordered Spaces. \textit{Annals of Probability} \textbf{5}, 899--912. 

\bibitem{M01} \textsc{M. N. Malyshkin} (2001). Subexponential Estimates of the Rate of Convergence to the Invariant Measure for Stochastic Differential Equations. \textit{Theory of Probability and its Applications} \textbf{45}, 466--479.

\bibitem{MT93a} \textsc{Sean P. Meyn, Richard L. Tweedie} (1993). Stability of Markovian Processes II. Continuous-Time Processes and Sampled Chains. \textit{Advances in Applied Probability} \textbf{25}, 487--517. 

\bibitem{MT93b} \textsc{Sean P. Meyn, Richard L. Tweedie} (1993). Stability of Markovian Processes III. Foster-Lyapunov Criteria for Continuous-Time Markov Processes. \textit{Advances in Applied Probability} \textbf{25}, 518--548. 

\bibitem{MT94} \textsc{Sean P. Meyn, Richard L. Tweedie} (1994). Computable Bounds for Geometric Convergence Rates of Markov Chains. \textit{Annals of Applied Probability} \textbf{4}, 981--1011. 

\bibitem{MTBook} \textsc{Sean P. Meyn, Richard L. Tweedie} (2009). \textit{Markov Chains and Stochastic Stability} Cambridge University Press, second edition. 

\bibitem{RT00} \textsc{Gareth O. Roberts, Richard L. Tweedie} (2000). Rates of Convergence of Stochastically Monotone and Continuous-Time Markov Chains. \textit{Journal of Applied Probability} \textbf{37}, 359--373. 

\bibitem{Rockner} \textsc{Michael R\"ockner, Feng-Yu Wang} (2001). Weak Poincare Inequalities and $L^2$ Convergence Rates of Markov Semigroups. \textit{Journal of Functional Analysis} \textbf{185} (2), 564--603.  

\bibitem{MCMC} \textsc{Jeffrey S. Rosenthal} (1995). Minorization Conditions and Convergence Rates for Markov Chain Monte Carlo. \textit{Journal of the American Statistical Association} \textbf{90}, 558--566. 

\bibitem{Sar16} \textsc{Andrey Sarantsev} (2016). Explicit Rates of Exponential Convergence for Reflected Jump-Diffusions on the Half-Line. \textit{Latin American Journal of Probability and Mathematical Statistics} \textbf{13}, 1069--1093. 

\bibitem{MyOwn10} \textsc{Andrey Sarantsev} (2017). Reflected Brownian Motion in a Convex Polyhedral Cone: Tail Estimates for the Stationary Distribution. \textit{Journal of Theoretical Probability} \textbf{30}, 1200--1223.

\bibitem{Sar} \textsc{Andrey Sarantsev} (2019). Convergence Rate to Equilibrium in Wasserstein Distance for Reflected Jump-Diffusions (2020). \textit{Statistics and Probability Letters} \textbf{165} 108860.

\bibitem{Skorohod} \textsc{Anatolii V. Skorohod} (1961). Stochastic Equations for Diffusion Processes in a Bounded Region. \textit{Theory of Probability and Its Applications} \textbf{6}, 264--274. 

\bibitem{TT94} \textsc{Pekka Tuominen, Richard L. Tweedie} (1994). Subgeometric Rates of Convergence of $f$-Ergodic Markov Chains. \textit{Advances in Applied Probability} \textbf{26}, 775--798.


\bibitem{WangBook} \textsc{Feng-Yu Wang} (2006). \textit{Functional Inequalities, Markov Semigroups and Spectral Theory}. Springer 

\bibitem{Zeif91} \textsc{Alexander I. Zeifman} (1991). Some Estimates of the Rate of Convergence for Birth and Death Processes. \textit{Journal of Applied Probability} \textbf{28}, 268--277. 

\end{thebibliography}
\end{document}